\documentclass{amsart}
\usepackage{graphicx, amsmath, amssymb}
\usepackage{color}

\makeatletter \oddsidemargin.9375in \evensidemargin \oddsidemargin
\def\Corresponding author{$^{*}$\protect\footnotetext{$^{*}$ C\lowercase{orresponding author.}}}
\def\authorsaddresses#1{\dedicatory{#1}}

\newtheorem{thm}{Theorem}[section]
\theoremstyle {definition}
\newtheorem{cor}[thm]{Corollary}
\newtheorem{prop}[thm]{Proposition}

\newtheorem{lem}[thm]{Lemma}

\newtheorem{rem}[thm]{Remark}

\numberwithin{equation}{section}

\begin{document}
\setcounter{page}{1}

\title[Domination number in the annihilating-submodule graph]{Domination number in the annihilating-submodule graph
 of modules over commutative rings}

\author[H. Ansari-Toroghy and S. Habibi]{H. Ansari-Toroghy$^1$ and S. Habibi$^2$}

\authorsaddresses{$^1$ Department of pure Mathematics,\\ Faculty of mathematical Sciences,\\ University of Guilan,
P. O. Box 41335-19141, Rasht, Iran.\\
e-mail: ansari@guilan.ac.ir\\
\vspace{0.5cm} $^2$ School of Mathematics, Institute for Research
in Fundamental Sciences (IPM), P.O. Box: 19395-5746, Tehran,
Iran.\\ Department of pure Mathematics, Faculty of mathematical
Sciences, University of Guilan, P. O. Box 41335-19141, Rasht,
Iran.
\\
e-mail: habibishk@gmail.com} 
\keywords{Commutative rings, annihilating-submodule graph,
domination number.\\This research was in part supported by a grant
from IPM (No. 96130028)}
\begin{abstract}
Let $M$ be a module over a commutative ring $R$. The
annihilating-submodule graph of $M$, denoted by $AG(M)$, is a
simple graph in which a non-zero submodule $N$ of $M$ is a vertex
if and only if there exists a non-zero proper submodule $K$ of $M$
such that $NK=(0)$, where $NK$, the product of $N$ and $K$, is
denoted by $(N:M)(K:M)M$ and two distinct vertices $N$ and $K$ are
adjacent if and only if $NK=(0)$. This graph is a submodule
version of the annihilating-ideal graph and under some conditions,
is isomorphic with an induced subgraph of the Zariski
topology-graph $G(\tau_T)$ which was introduced in (The Zariski
topology-graph of modules over commutative rings, Comm. Algebra.,
42 (2014), 3283--3296). In this paper, we study the domination
number of $AG(M)$ and some connections between the graph-theoretic
properties of $AG(M)$ and algebraic properties of module $M$.
\end{abstract}
\maketitle
\section{Introduction}
 Throughout this paper $R$ is a commutative ring with a non-zero
identity and $M$ is a unital $R$-module. By $N\leq M$ (resp. $N<
M$) we mean that $N$ is a submodule (resp. proper submodule) of
$M$.

Define $(N:_{R}M)$ or simply $(N:M)=\{r\in R|$ $rM\subseteq N\}$
for any $N\leq M$. We denote $((0):M)$ by $Ann_{R}(M)$ or simply
$Ann(M)$. $M$ is said to be faithful if $Ann(M)=(0)$. Let $N,
K\leq M$. Then the product of $N$ and $K$, denoted by $NK$, is
defined by $(N:M)(K:M)M$ (see \cite{af07}). Define $ann(N)$ or
simply $annN=\{m\in M|$ $m(K:M)=0\}$.

The prime spectrum of $M$ is the set of all prime submodules of
$M$ and denoted by $Spec(M)$, $Max(M)$ is the set of all maximal
submodules of $M$, and $J(M)$, the jacobson radical of $M$, is the
intersection of all elements of $Max(M)$, respectively.

There are many papers on assigning graphs to rings or modules
(see, for example, \cite{al99, ah14, b88, br11}). The
annihilating-ideal graph $AG(R)$ was introduced and studied in
\cite{br11}. $AG(R)$ is a graph whose vertices are ideals of $R$
with nonzero annihilators and in which two vertices $I$ and $J$
are adjacent if and only if $IJ=(0)$. Later, it was modified and
further studied by many authors (see \cite{aa12, aa13, aa14, nmk,
ts}).

In \cite{ah14}, the present authors introduced and studied the
graph $G(\tau_T)$ (resp. $AG(M)$), called the \textit{Zariski
topology-graph} (resp. \textit {the annihilating-submodule
graph}), where $T$ is a non-empty subset of $Spec(M)$.

$AG(M)$ is an undirected graph with vertices $V(AG(M))$= $\{N \leq
M |$ there exists $(0)\neq K<M$ with $NK=(0)$\}. In this graph,
distinct vertices $N,L \in V(AG(M))$ are adjacent if and only if
$NL=(0)$ (see \cite{ah16, ah160}). Let $AG(M)^{*}$ be the subgraph
of $AG(M)$ with vertices $V(AG(M)^{*})=\{ N<M$ with $(N:M)\neq
Ann(M)|$ there exists a submodule $K<M$ with $(K:M)\neq Ann(M)$
and $NK=(0)\}$. By \cite[Theorem 3.4]{ah14}, one conclude that
$AG(M)^{*}$ is a connected subgraph. Note that $M$ is a vertex of
$AG(M)$ if and only if there exists a nonzero proper submodule $N$
of $M$ with $(N:M)=Ann(M)$ if and only if every nonzero submodule
of $M$ is a vertex of $AG(M)$. Clearly, if $M$ is not a vertex of
$AG(M)$, then $AG(M)=AG(M)^{*}$. In \cite[Lemma 2.8]{ah140}, we
showed that under some conditions, $AG(M)$ is isomorphic with an
induced subgraph of the Zariski topology-graph $G(\tau_T)$.

In this paper, we study the domination number of $AG(M)$ and some
connections between the graph-theoretic properties of $AG(M)$ and
algebraic properties of module $M$.

A prime submodule of $M$ is a submodule $P\neq M$ such that
whenever $re\in P$ for some
  $r\in R$ and $e \in M$, we have $r\in (P:M)$ or $e\in P$ \cite{lu84}.

The notations $Z(R)$ and $Nil(R)$ will denote the set of all
zero-divisors, the set of all nilpotent elements of $R$,
respectively. Also, $Z_{R}(M)$ or simply $Z(M)$, the set of zero
divisors on $M$, is the set $\{r\in R|$ $rm=0$ for some $0\neq
m\in M \}$. If $Z(M)=0$, then we say that $M$ is a domain. An
ideal $I\leq R$ is said to be nil if $I$ consist of nilpotent
elements.

Let us introduce some graphical notions and denotations that are
used in what follows: A graph $G$ is an ordered triple $(V(G),
E(G), \psi_G )$ consisting of a nonempty set of vertices,
 $V(G)$, a set $E(G)$ of edges, and an incident function $\psi_G$ that associates an
 unordered pair of distinct vertices with each edge. The edge $e$ joins $x$
  and $y$ if $\psi_G(e)=\{x, y\}$, and we
 say $x$ and $y$ are adjacent. The number of edges incident at $x$
 in $G$ is called the degree of the vertex $x$ in $G$ and is
 denoted by $d_G(v)$ or simply $d(v)$.
 A path in graph $G$ is a finite sequence of vertices $\{x_0,
x_1,\ldots ,x_n\}$, where $x_{i-1}$ and $x_i$ are adjacent for
each $1\leq i\leq n$ and we denote $x_{i-1} - x_i$ for existing an
edge between $x_{i-1}$ and $x_i$. The distance between two
vertices $x$ and $y$, denoted $d(x, y)$, is the length of the
shortest path from $x$ to $y$. The diameter of a connected graph
$G$ is the maximum distance between two distinct vertices of $G$.
For any vertex $x$ of a connected graph $G$, the eccentricity of
$x$, denoted $e(x)$, is the maximum of the distances from $x$ to
the other vertices of $G$. The set of vertices with minimum
eccentricity is called the center of the graph $G$, and this
minimum eccentricity value is the radius of $G$. For some
$U\subseteq V(G)$, we denote by $N(U)$, the set of all vertices of
$G\setminus U$ adjacent to at least one vertex of $U$ and
$N[U]=N(U)\cup \{U\}$.

 A graph $H$ is a subgraph of $G$, if $V(H)\subseteq V(G)$,
$E(H)\subseteq E(G)$, and $\psi_H$ is the restriction of $\psi_G$
to $E(H)$. A subgraph $H$ of $G$ is a spanning subgraph of $G$ if
$V(H)=V(G)$. A spanning subgraph $H$ of $G$ is called a perfect
matching of $G$ if every vertex of $G$ has degree 1.

 A clique of a graph is a complete subgraph and the supremum of the sizes of
cliques in $G$, denoted by $cl(G)$, is called the clique number of
$G$. Let $\chi(G)$ denote the chromatic number of the graph $G$,
that is, the minimal number of colors needed to color the vertices
of $G$ so that no two adjacent vertices have the same color.
Obviously $\chi(G)\geq cl(G)$.

 A subset $D$ of $V(G)$ is called a
dominating set if every vertex of $G$ is either in $D$ or adjacent
to at least one vertex in $D$. The domination number of $G$,
denoted by $\gamma(G)$, is the number of vertices in a smallest
dominating set of $G$. A total dominating set of a graph $G$ is a
set $S$ of vertices of $G$ such that every vertex is adjacent to a
vertex in $S$. The total domination number of $G$, denoted by
$\gamma_t(G)$, is the minimum cardinality of a total dominating
set. A dominating set of cardinality $\gamma(G)$ ($\gamma_t(G)$)
is called a $\gamma$-set ($\gamma_t$-set). A dominating set $D$ is
a connected dominating set if the subgraph $<D>$ induced by $D$ is
a connected subgraph of $G$. The connected domination number of
$G$, denoted by $\gamma_c(G)$, is the minimum cardinality of a
connected dominating set of $G$. A dominating set $D$ is a clique
dominating set if the subgraph $<D>$ induced by $D$ is complete in
$G$. The clique domination number $\gamma_{cl}(G)$ of $G$ equals
the minimum cardinality of a clique dominating set of $G$. A
dominating set $D$ is a paired-dominating set if the subgraph
 $<D>$ induced by $D$ has a perfect matching. The paired-domination number
$\gamma_{pr}(G)$ of $G$ equals the minimum cardinality of a
paired-dominating set of $G$.

 A vertex $u$ is a neighbor of $v$ in $G$, if $uv$ is an edge of
 $G$, and $u\neq v$. The set of all neighbors of $v$ is the open
 neighborhood of $v$ or the neighbor set of $v$, and is denoted by
 $N(v)$; the set $N[v]=N(v)\cup \{v\}$ is the closed neighborhood
 of $v$ in $G$.

  Let $S$ be a dominating set of a graph $G$, and $u\in S$. The
  private neighborhood of $u$ relative to $S$ in $G$ is the set of
  vertices which are in the closed neighborhood of $u$, but not in
  the closed neighborhood of any vertex in $S\setminus
  \{u\}$. Thus the private neighborhood $P_N(u, S)$ of $u$ with
  respect to $S$ is given by $P_N(u, S)=N[u]\setminus (\cup_{v\in S\setminus \{u\}}
  N[v])$. A set $S\subseteq V(G)$ is called irredundant if every
  vertex $v$ of $S$ has at least one private neighbor. An irredundant set $S$ is a
   maximal irredundant set if for every
vertex $u \in V\setminus S$, the set $S\cup \{u\}$ is not
irredundant. The irredundance number $ir(G)$ is the minimum
cardinality of maximal irredundant sets. There are so many
domination parameters in the literature and for more details one
can refer \cite{hhs}.

 A bipartite graph is a graph whose vertices can be
divided into two disjoint sets $U$ and $V$ such that every edge
connects a vertex in $U$ to one in $V$; that is, $U$ and $V$ are
each independent sets and complete bipartite graph on $n$ and $m$
vertices, denoted by $K_{n, m}$, where $V$ and $U$ are of size $n$
and $m$, respectively, and $E(G)$ connects every vertex in $V$
with all vertices in $U$. Note that a graph $K_{1, m}$ is called a
star graph and the vertex in the singleton partition is called the
center of the graph. We denote by $P_{n}$ a path of order $n$ (see
\cite{r05}).

In section 2, a dominating set of $AG(M)$ is constructed using
elements of the center when $M$ is an Artinian module. Also we
prove that the domination number of $AG(M)$ is equal to the number
of factors in the Artinian decomposition of $M$ and we also find
several domination parameters of $AG(M)$. In section 3, we study
the domination number of the annihilating-submodule graphs for
reduced rings with finitely many minimal primes and faithful
modules. Also, some relations between the domination numbers and
the total domination numbers of annihilating-submodule graphs are
studied.

 The following results are useful for further reference in this
paper.

\begin{prop}\label{p1.1} Suppose that $e$ is an idempotent element of
$R$. We have the following statements.

\begin {itemize}
\item [(a)] $R=R_{1}\times R_{2}$, where $R_{1}=eR$ and
$R_{2}=(1-e)R$. \item [(b)] $M=M_{1}\times M_{2}$, where
$M_{1}=eM$ and $M_{2}=(1-e)M$. \item [(c)] For every submodule $N$
of $M$, $N=N_{1}\times N_{2}$ such that $N_{1}$ is an
$R_{1}$-submodule $M_{1}$, $N_{2}$ is an $R_{2}$-submodule
$M_{2}$, and $(N:_{R}M)=(N_{1}:_{R_{1}}M_{1})\times
(N_{2}:_{R_{2}}M_{2})$. \item [(d)] For submodules $N$ and $K$ of
$M$, $NK=N_{1}K_{1} \times N_{2}K_{2}$ such that $N=N_{1}\times
N_{2}$ and $K=K_{1}\times K_{2}$. \item [(e)]  Prime submodules of
$M$ are $P\times M_{2}$ and $M_{1}\times Q$, where $P$ and $Q$ are
prime submodules of $M_{1}$ and $M_{2}$, respectively.
\end{itemize}

\end{prop}

\begin{proof}
This is clear.
\end{proof}

We need the following results.

\begin{lem}\label{l1.2} (See \cite[Proposition 7.6]{af74}.)
Let $R_{1}, R_{2}, \ldots , R_{n}$ be non-zero ideals of $R$. Then
the following statements are equivalent:

\begin{itemize}
\item [(a)] $R= R_{1} \times \ldots \times R_{n}$; \item [(b)] As
an abelian group $R$ is the direct sum of $ R_{1}, \ldots ,
R_{n}$; \item [(c)] There exist pairwise orthogonal idempotents
$e_{1},\ldots, e_{n}$ with $1=e_{1}+ \ldots +e_{n}$, and
$R_{i}=Re_{i}$, $i=1, \ldots ,n$.
\end{itemize}

\end{lem}

\begin{lem}\label{l1.3} (See \cite[Theorem 21.28]{l91}.)
Let $I$ be a nil ideal in $R$ and $u\in R$ be such that
 $u+I$ is an idempotent in $R/I$. Then there exists an idempotent
 $e$ in $uR$ such that $e-u\in I$.
\end{lem}

\begin{lem}\label{l1.4} (See \cite[Lemma 2.4]{ah16}.)
Let $N$ be a minimal submodule of $M$ and let $Ann(M)$ be a nil
ideal. Then we have $N^{2}=(0)$ or $N=eM$ for some idempotent
$e\in R$.
\end{lem}

\begin{prop}\label{p1.5} Let $R/Ann(M)$ be an Artinian ring and let $M$ be
 a finitely generated module.
Then every nonzero proper submodule $N$ of $M$ is a vertex in
$AG(M)$.
\end{prop}

\begin{thm}\label{t1.6} (See \cite[Theorem 2.5]{ah16}.) Let $Ann(M)$ be a nil ideal.
There exists a vertex of $AG(M)$ which is adjacent to every other
vertex if and only if $M=eM\oplus (1-e)M$, where $eM$ is a simple
module and $(1-e)M$ is a prime module for some idempotent $e\in
R$, or $Z(M)=Ann((N:M)M)$, where $N$ is a nonzero proper submodule
of $M$ or $M$ is a vertex of $AG(M)$.
\end{thm}

\begin{thm}\label{t1.7} (See \cite[Theorem 3.3]{ah16}.) Let $M$ be a faithful module.
Then the following statements are equivalent.

\begin{itemize}
\item [(a)] $\chi(AG(M)^{*})=2$. \item [(b)] $AG(M)^{*}$ is a
bipartite graph with two nonempty parts. \item [(c)] $AG(M)^{*}$
is a complete bipartite graph with two nonempty parts. \item [(d)]
Either $R$ is a reduced ring with exactly two minimal prime
ideals, or $AG(M)^{*}$ is a star graph with more than one vertex.
\end{itemize}

\end{thm}

\begin{cor}\label{c1.8} (See \cite[Corollary 3.5]{ah16}.) Let $R$ be a reduced ring and assume
 that $M$ is a faithful module.
Then the following statements are equivalent.

\begin{itemize}
\item [(a)] $\chi(AG(M)^{*})=2$. \item [(b)] $AG(M)^{*}$ is a
bipartite graph with two nonempty parts. \item [(c)] $AG(M)^{*}$
is a complete bipartite graph with two nonempty parts. \item [(d)]
$R$ has exactly two minimal prime ideals.
\end{itemize}

\end{cor}

\begin{prop}\label{p1.9} (See \cite[Proposition 3.9]{hhs}.) Every minimal dominating set
in a graph $G$ is a maximal irredundant set of $G$.
\end{prop}

\section{Domination number in the annihilating-submodule graph for Artinian modules}

The main goal in this section, is to obtain the value certain
domination parameters of the annihilating-submodule graph for
Artinian modules.

Recall that $M$ is a vertex of $AG(M)$ if and only if there exists
a nonzero proper submodule $N$ of $M$ with $(N:M)=Ann(M)$ if and
only if every nonzero submodule of $M$ is a vertex of $AG(M)$. In
this case, the vertex $N$ is adjacent to every other vertex. Hence
$\gamma(AG(M))=1=\gamma_t((AG(M)))$.  So we assume that
\textbf{throughout this paper $M$ is not a vertex of $AG(M)$}.
Clearly, if $M$ is not a vertex of $AG(M)$, \textbf{then
$AG(M)=AG(M)^{*}$.}

We start with the following remark which completely characterizes
all modules for which $\gamma((AG(M))) = 1$.

\begin{rem}\label{r2.1} Let $Ann(M)$ be a nil ideal. By Theorem
\ref{t1.6}, there exists a vertex of $AG(M)$ which is adjacent to
every other vertex if and only if $M=eM\oplus (1-e)M$, where $eM$
is a simple module and $(1-e)M$ is a prime module for some
idempotent $e\in R$, or $Z(M)=Ann((N:M)M)$, where $N$ is a nonzero
proper submodule of $M$ or $M$ is a vertex of $AG(M)$. Now, let
$Ann(M)$ be a nil ideal and $M$ be a domain module. Then
$\gamma((AG(M))) = 1$ if and only if $M=eM\oplus (1-e)M$, where
$eM$ is a simple module and $(1-e)M$ is a prime module for some
idempotent $e\in R$.
\end{rem}

\begin{thm}\label{t2.2} Let $M$ be a f.g Artinian local module. Assume
that $N$ is the unique maximal submodule of $M$. Then the radius
of $AG(M)$ is $0$ or $1$ and the center of $AG(M)$ is
$\{K\subseteq ann(N)| K\neq (0)$ is a submodule in $M\}$.
\end{thm}

\begin{proof}

If $N$ is the only non-zero proper submodule of $M$, then
$AG(M)\cong K_1$, $e(N) = 0$ and the radius of $AG(M)$ is $0$.
Assume that the number of non-zero proper submodules of $M$ is
greater than $1$. Since $M$ is f.g Artinian module, there exists
$m\in \Bbb N$, $m > 1$ such that $N^m = (0)$ and $N^{m-1}\neq
(0)$. For any non-zero submodule $K$ of $M$, $KN^{m-1}\subseteq
NN^{m-1} = (0)$ and so $d(N^{m-1}, K) = 1$. Hence $e(N^{m-1}) = 1$
and so the radius of $AG(M)$ is $1$. Suppose $K$ and $L$ are
arbitrary non-zero submodules of $M$ and $K\subseteq ann(N)$. Then
$KL\subseteq KN = (0)$ and hence $e(K) = 1$. Suppose $(0)\neq K'
\nsubseteq ann(N)$. Then $K'N\neq (0)$ and so $e(K') > 1$. Hence
the center of $AG(M)$ is $\{K\subseteq ann(N)| K\neq (0)$ is a
submodule in $M\}$.

\end{proof}

\begin{cor}\label{c2.3}
Let $M$ be a f.g Artinian local module and $N$ is the unique
maximal submodule of $M$. Then the following hold good.

\begin{itemize}
\item [(a)] $\gamma(AG(M))=1$.

\item [(b)] $D$ is a $\gamma$-set of $AG(M)$ if and only if
$D\subseteq ann(N)$.
 \end{itemize}
\end{cor}

\begin{proof}
 $(a)$ Trivial from Theorem \ref{t2.6}.\\
 $(b)$ Let $D = \{K\}$ be a $\gamma$-set of $AG(M)$. Suppose $K\nsubseteq ann(N)$.
  Then $KN\neq (0)$
and so $N$ is not dominated by $K$, a contradiction. Conversely,
suppose $D\subseteq ann(N)$. Let $K$ be an arbitrary vertex in
$AG(M)$. Then $KL\subseteq NL = (0)$ for every $L\in D$. i.e.,
every vertex $K$ is adjacent to every $L\in D$. If $|D| > 1$, then
$D\setminus \{L'\}$ is also a dominating set of $AG(M)$ for some
$L'\in D$ and so $D$ is not minimal. Thus $|D| = 1$ and so $D$ is
a $\gamma$-set by $(a)$.
\end{proof}

\begin{thm}\label{t2.4}
Let $M=\oplus _{i=1}^n M_i$, where $M_i$ is a f.g Artinian local
module for all $1\leq i\leq n$ and $n\geq 2$. Then the radius of
$AG(M)$ is $2$ and the center of $AG(M)$ is $\{K\subseteq J(M)|
K\neq (0)$ is a submodule in $M \}$.
\end{thm}

\begin{proof}
Let $M=\oplus _{i=1}^n M_i$, where $M_i$ is a f.g Artinian local
module for all $1\leq i\leq n$ and $n\geq 2$. Let $J_i$ be the
unique maximal submodule in $M_i$ with nilpotency $n_i$. Note that
$Max(M) = \{N_1,\ldots ,N_n| N_i = M_1 \oplus \ldots \oplus
M_{i-1} \oplus J_i\oplus M_{i+1}\oplus \ldots \oplus M_n, 1\leq
i\leq n\}$ is the set of all maximal submodules in $M$. Consider
$D_i = (0) \oplus \ldots \oplus (0) \oplus J_i^{n_i-1}\oplus
(0)\oplus \ldots \oplus (0)$ for $1\leq i\leq n$. Note that $J(M)
= J_1\oplus \ldots \oplus J_n$ is the Jacobson radical of $M$ and
any non-zero submodule in $M$ is adjacent to $D_i$ for some $i$.
Let $K$ be any non-zero submodule of $M$. Then
$K=\oplus _{i=1}^n K_i$, where $K_i$ is a submodule of $M_i$.\\
\textbf{Case 1}. If $K = N_i$ for some $i$, then $KD_j\neq (0)$
and $KN_j\neq (0)$ for all $j\neq i$. Note that $N(K)=\{(0) \oplus
\ldots \oplus (0) \oplus L_i\oplus (0)\oplus \ldots \oplus (0)|
J_iL_i= (0)$, $L_i$ is a nonzero submodule in $M_i \}$. Clearly
$N(K)\cap N(N_j) = (0)$, $d(K,N_j)\neq 2$ for all $j\neq i$, and
so $K- D_i - D_j
 - N_j$ is a path in $AG(M)$. Therefore $e(K) = 3$ and so $e(N) = 3$ for
all $N\in Max(M)$.\\
\textbf{Case 2}. If $K\neq D_i$ and $K_i \subseteq J_i$ for all
$i$. Then $KD_i = (0)$ for all $i$. Let $L$ be any non-zero
submodule of $M$ with $KL\neq (0)$. Then $LD_j = (0)$ for some
$j$, $K - D_j - L$ is a path in $AG(M)$ and so $e(K) = 2$.\\
\textbf{Case 3}. If $K_i = M_i$ for some $i$, then $KD_i\neq (0)$,
$KN_i \neq (0)$ and $KD_j = (0)$ for some $j\neq i$. Thus $K - D_j
- D_i - N_i$ is a path in $AG(M)$, $d(K,N_i) = 3$ and so $e(K) =
3$. Thus $e(K) = 2$ for all $K\subseteq J(M)$. Further note that
in all the cases center of $AG(M)$ is $\{K\subseteq J(M)| K\neq
(0)$ is a submodule in $M \}$.
\end{proof}

In view of Theorems \ref{t2.2} and \ref{t2.4}, we have the
following corollary.

\begin{cor}\label{c2.5}
Let $M=\oplus _{i=1}^n M_i$, where $M_i$ is a simple module for
all $1\leq i\leq n$ and $n\geq 2$. Then the radius of $AG(M)$ is
$1$ or $2$ and the center of $AG(M)$ is $\cup_{i=1}^n D_i$, where
$D_i = (0) \oplus \ldots \oplus (0) \oplus M_i\oplus (0)\oplus
\ldots \oplus (0)$ for $1\leq i\leq n$.
\end{cor}

\begin{thm}\label{t2.6}
Let $M=\oplus _{i=1}^n M_i$, where $M_i$ is a f.g Artinian local
module for all $1\leq i\leq n$ and $n\geq 2$. Then
$\gamma(AG(M))=n$.
\end{thm}

\begin{proof}
Let $N_i$ be the unique maximal submodule in $M_i$ with nilpotency
$n_i$. Let $\Omega= \{D_1, D_2, \ldots ,D_n\}$, where $D_i = (0)
\oplus \ldots \oplus (0) \oplus J_i^{n_i -1}\oplus (0)\oplus
\ldots \oplus (0)$ for $1\leq i\leq n$. Note that any non-zero
submodule in $M$ is adjacent to $D_i$ for some $i$. Therefore
$N[\Omega] = V(AG(M))$, $\Omega$
 is a dominating set of $AG(M)$ and so $\gamma(AG(M))\leq n$.
 Suppose $S$ is a dominating set of $AG(M)$ with $|S| < n$.
Then there exists $N\in Max(M)$ such that $NK\neq (0)$ for all
$K\in S$, a contradiction. Hence $\gamma(AG(M))=n$. \end{proof}

In view of Theorem \ref{t2.6}, we have the following corollary.

\begin{cor}\label{c2.7}
Let $M=\oplus _{i=1}^n M_i$, where $M_i$ is a f.g Artinian local
module for all $1\leq i\leq n$ and $n\geq 2$. Then

\begin{itemize}
\item [(a)] $ir(AG(M))=n$.

\item [(b)] $\gamma_c(AG(M))=n$.

\item [(c)] $\gamma_t(AG(M))=n$.

\item [(d)] $\gamma_{cl}(AG(M))=n$.

\item [(e)] $\gamma_{pr}(AG(M))=n$, if $n$ is even and
$\gamma_{pr}(AG(M))=n+1$, if $n$ is odd.

 \end{itemize}
\end{cor}

\begin{proof}

Consider the $\gamma$-set of $AG(M)$ identified in the proof of
Theorem \ref{t2.6}. By Proposition \ref{p1.9}, $\Omega$
 is a maximal irredundant set with minimum
cardinality and so $ir(AG(M))=n$. Clearly $<\Omega>$ is a complete
subgraph of $AG(M)$. Hence
$\gamma_c(AG(M))=\gamma_t(AG(M))=\gamma_{cl}(AG(M))=n$. If $n$ is
even, then $<\Omega>$ has a perfect matching and so $\Omega$
 is a paired dominating set of $AG(M)$. Thus $pr(AG(M)) = n$. If $n$ is odd, then
 $<\Omega \cup K>$ has a perfect matching for some $K\in V(AG(M))\setminus \Omega$.
 and so  $\Omega \cup {K}$ is a paired dominating set of $AG(M)$. Thus $\gamma_{pr}(AG(M))=n$ if
$n$ even and $\gamma_{pr}(AG(M))=n+1$ if $n$ is odd.

\end{proof}

Let $M=\oplus _{i=1}^n M_i$, where $M_i$ is a f.g Artinian local
module for all $1\leq i\leq n$ and $n\geq 2$. Then by Theorem
\ref{t2.4}, radius of $AG(M)$ is $2$. Further, by Theorem
\ref{t2.6}, the domination number of $AG(M)$ is equal to $n$,
where $n$ is the number of distinct maximal submodules of $M$.
However, this need not be true if the radius of $AG(M)$ is $1$.
For, consider the ring $M = M_1 \oplus M_2$, where $M_1$ and $M_2$
are simple modules. Then $AG(M)$ is a star graph and so has radius
$1$, whereas $M$ has two distinct maximal submodules. The
following corollary shows that a more precise relationship between
the domination number of $AG(M)$ and the number of maximal
submodules in $M$, when $M$ is finite.

\begin{cor}\label{c2.8}
Let $M$ be a finite module and $\gamma((AG(M))) = n$. Then either
$M = M_1 \oplus M_2$, where $M_1$, $M_2$ are simple modules or $M$
has n maximal submodules.
\end{cor}

\begin{proof}
When $\gamma((AG(M))) = 1$, proof follows from \cite[Corollary
2.12]{ah16}. When $\gamma((AG(M))) = n$, then $M$ cannot be $M =
M_1 \oplus M_2$, where $M_1$, $M_2$ are simple modules. Hence
$M=\oplus _{i=1}^m M_i$, where $M_i$ is a f.g Artinian local
module for all $1\leq i\leq m$ and $m\geq 2$. By Theorem
\ref{t2.6}, $\gamma((AG(M))) = m$. Hence by assumption $m = n$.
i.e., $M=\oplus _{i=1}^n M_i$, where $M_i$ is a f.g Artinian local
module for all $1\leq i\leq n$ and $n\geq 2$. One can see now that
$M$ has $n$ maximal submodules.
\end{proof}

\section{The relationship between $\gamma_t((AG(M)))$ and $\gamma((AG(M)))$}

The main goal in this section is to study the relation between
$\gamma_t((AG(M)))$ and $\gamma((AG(M))) $.

\begin{thm}\label{t3.1} Let $M$ be a module. Then

$\gamma_t((AG(M)))= \gamma((AG(M)))$ or $\gamma_t((AG(M)))=
\gamma((AG(M)))+1$.
 \end{thm}

\begin{proof}
Let $\gamma_t((AG(M)))\neq \gamma((AG(M)))$ and $D$ be a
$\gamma$-set of $AG(M)$. If  $\gamma((AG(M)))=1$, then it is clear
that $\gamma_t((AG(M)))=2$. So let $\gamma((AG(M)))> 1$ and put $k
= Max \{n|$there exist $L_1, \ldots , L_n \in D$ ; $\sqcap_{i=1}
^n L_i \neq 0 \}$. Since $\gamma_t((AG(M)))\neq \gamma((AG(M)))$,
we have $k \geq 2$. Let $L_1, \ldots , L_k \in D$ be such that
$\sqcap_{i=1} ^k L_i \neq 0$. Then $S = \{\sqcap_{i=1} ^k L_i, ann
L_1, \ldots , ann L_k \}\cup D\setminus \{L_1, \ldots , L_k \}$
is a $\gamma_t$-set. Hence $\gamma_t((AG(M)))= \gamma((AG(M)))+1$.

\end{proof}

In the following result we find the total domination number of
$AG(M)$.

 \begin{thm}\label{t3.2} Let $S$ be the set of all maximal elements of
the set $V(AG(M))$. If $|S| > 1$, then $\gamma_t((AG(M))) = |S|$.
\end{thm}

\begin{proof}
Let $S$ be the set of all maximal elements of the set $V(AG(M))$,
$K\in S$ and $|S| > 1$. First we show that $K = ann(ann K)$ and
there exists $m\in M$ such that $K = ann(m)$. Let $K\in S$. Then
$ann K\neq 0$ and so there exists $0\neq m\in ann K$. Hence
$K\subseteq ann(ann K)\subseteq ann(m)$. Thus by the maximality of
$K$, we have $K = ann(ann K) = ann(m)$. By Zorn' Lemma it is clear
that if $V(AG(M))\neq \emptyset$, then $S\neq \emptyset$. For any
$K\in S$ choose $m_K \in M$ such that $K = ann(m_K)$. We assert
that $D = \{Rm_K | K\in S\}$ is a total dominating set of $AG(M)$.
Since for every $L\in V(AG(M))$ there exists $K\in S$ such that
$L\subseteq K = ann(m_K)$, $L$ and $Rm_K$ are adjacent. Also for
each pair $K, K'\in S$, we have $(Rm_K)(Rm_{K'}) = 0$. Namely, if
there exists $m\in (Rm_K)(Rm_{K'})\setminus \{0\}$, then $K = K' =
ann(m)$. Thus $\gamma_t((AG(M)))\leq |S|$. To complete the proof,
we show that each element of an arbitrary $\gamma_t$-set of
$AG(M)$ is adjacent to exactly one element of $S$. Assume to the
contrary, that a vertex $L'$ of a $\gamma_t$-set of $AG(M)$ is
adjacent to $K$ and $K'$, for $K, K' \in S$. Thus $K = K' = ann
L'$, which is impossible. Therefore $\gamma_t((AG(M))) = |S|$.
\end{proof}

\begin{thm}\label{t3.3} Let $R$ be a reduced ring, $M$ is a faithful module,
 and $|Min(R)| < \infty$. If $\gamma((AG(M)))> 1$, then $\gamma_t((AG(M)))= \gamma((AG(M)))= |Min(R)|$.
\end{thm}

 \begin{proof} Since $R$ is reduced, $M$ is a faithful module,
 and $\gamma((AG(M)))> 1$, we have $|Min(R)| > 1$. Suppose that
$Min(R) = \{p_1, \ldots , p_n \}$. If $n = 2$, the result follows
from Corollary \ref{c1.8}. Therefore, suppose that $n \geq 3$.
Define $\widehat{p_iM} = p_1 \ldots p_{i-1}p_{i+1} \ldots p_n M$,
for every $i = 1, \ldots , n$. Clearly, $\widehat{p_iM}\neq 0$,
for every $i = 1, \ldots , n$. Since $R$ is reduced, we deduce
that $\widehat{p_iM} p_iM=0$. Therefore, every $p_i M$ is a vertex
of $AG(M)$. If $K$ is a vertex of $AG(M)$, then by \cite[Corollary
3.5]{ati69}, $(K:M)\subseteq Z(R) = \cup_{i=1} ^n p_i$. It follows
from the Prime Avoidance Theorem that $(K:M)\subseteq p_i$, for
some $i$, $1\leq i \leq n$. Thus $p_iM$ is a maximal element of
$V(AG(M))$, for every $i = 1, \ldots , n$. From Theorem
\ref{t3.2}, $\gamma_t((AG(M)))= |Min(R)|$. Now, we show that $
\gamma((AG(M)))= n$. Assume to the contrary, that $B = \{J_1,
\ldots , J_{n-1}\}$ is a dominating set for $AG(M)$. Since $n \geq
3$, the submodules $p_iM$ and $p_jM$ , for $i \neq j$ are not
adjacent (from $p_ip_j = 0 \subseteq p_k$ it would follow that
$p_i\subseteq p_k$, or $p_j\subseteq p_k$ which is not true).
Because of that, we may assume that for some $k < n - 1$, $J_i =
p_iM$ for $i = 1,\ldots, k$, but none of the other of submodules
from $B$ are equal to some $p_sM$ (if $B = \{p_1M, \ldots ,
p_{n-1}M\}$, then $p_nM$ would be adjacent to some $p_iM$, for
$i\neq n$). So, every submodule in $\{p_{k+1}M, . . . , p_nM\}$ is
adjacent to a submodule in $\{J_{k+1}, . . . , J_{n-1}\}$. It
follows that for some $s\neq t$, there is an $l$ such that
$(p_sM)J_l= 0 = (p_tM)J_l$. Since $p_s\nsubseteq p_t$, it follows
that $J_l\subseteq p_tM$, so $J_l ^2 = 0$, which is impossible,
since the ring $R$ is reduced. So $\gamma_t((AG(M)))=
\gamma((AG(M)))= |Min(R)|$.
\end{proof}

Theorem \ref{t3.3} leads to the following corollary.

\begin{cor}\label{c3.4} Let $R$ be a reduced ring, $M$ is a faithful module,
 and $|Min(R)| < \infty$, then the following are
equivalent.

\begin{itemize}
\item [(a)] $\gamma(AG(M))=2$.

\item [(b)] $AG(M)$ is a bipartite graph with two nonempty parts.

\item [(c)] $AG(M)$ is a complete bipartite graph with two
nonempty parts.

\item [(d)]  $R$ has exactly two minimal primes.

 \end{itemize}
\end{cor}

\begin{proof}
Follows from Theorem \ref{t3.3} and Corollary \ref{c1.8}.
\end{proof}

In the following theorem the domination number of bipartite
annihilating-submodule graphs is given.

\begin{thm}\label{t3.5} Let $M$ be a faithful module.
If $AG(M)$ is a bipartite graph, then $\gamma((AG(M)))\leq 2$.
\end{thm}

\begin{proof}
Let $M$ be a faithful module. If $AG(M)$ is a bipartite graph,
then from Theorem \ref{t1.7}, either $R$ is a reduced ring with
exactly two minimal prime ideals, or $AG(M)$ is a star graph with
more than one vertex. If $R$ is a reduced ring with exactly two
minimal prime ideals, then the result follows by Corollary
\ref{c3.4}. If $AG(M)$ is a star graph with more than one vertex,
then we are done.
\end{proof}

The next theorem is on the total domination number of the
annihilating-submodule graphs of Artinian modules.

\begin{thm}\label{t3.6} Let $M=\oplus _{i=1}^n M_i$, where $M_i$ is a f.g Artinian local
module for all $1\leq i\leq n$, $n\geq 2$, and $M \neq M_1 \oplus
M_2$, where $M_1$, $M_2$ are simple modules. Then
$\gamma_t((AG(M)))= \gamma((AG(M)))= |Min(R)|$.
\end{thm}

\begin{proof}

By Proposition \ref{p1.5}, every nonzero proper submodule of $M$
is a vertex in $AG(M)$. So, the set of maximal elements of
$V(AG(M))$ and $Max(M)$ are equal. Let $M=\oplus _{i=1}^n M_i$,
where $(M_i, J_i)$ is a f.g Artinian local module for all $1\leq
i\leq n$ and $n\geq 2$. Let $Max(M) = \{N_i = M_1 \oplus \ldots
\oplus M_{i-1} \oplus J_i\oplus M_{i+1}\oplus \ldots \oplus M_n |
1 \leq i \leq n \}$. By Theorem \ref{t3.2}, $\gamma_t((AG(M)))=
 |Max(M)|$. In the sequel, we prove that
$\gamma((AG(M))) = n$. Assume to the contrary, the set $\{K_1,
\ldots , K_{n-1}\}$ is a dominating set for $AG(M)$. Since $M \neq
M_1 \oplus M_2$, where $M_1$, $M_2$ are simple modules, we find
that $K_i N_s=K_i N_t=0$, for some $i, t, s$, where $1 \leq i \leq
n-1$ and $1 \leq t, s \leq n$. This means that $K_i = 0$, a
contradiction.
\end{proof}

The following theorem provides an upper bound for the domination
number of the annihilating-submodule graph of a Noetherian module.

\begin{thm}\label{t3.7}
If R is a Notherian ring and $M$ a f.g module, then
$\gamma((AG(M)))\leq |Ass(M)|< \infty$.
\end{thm}

\begin{proof}
 By \cite{s}, Since R is a Notherian ring and $M$ a f.g module, $|Ass(M)|< \infty$.
  Let $Ass(M) = \{p_1, . . . , p_n\}$ where
$p_i = ann(m_i)$ for some $m_i \in M$ for every $i = 1, \ldots ,
n$. Set $A = \{Rm_i | 1 \leq i \leq n \}$. We show that $A$ is a
dominating set of $AG(M)$. Clearly, every $Rm_i$ is a vertex of
$AG(M)$, for $i = 1, \ldots , n$ $( (p_iM)(m_iR)=0)$. If $K$ is a
vertex of $AG(M)$, then \cite[Corollary 9.36]{s} implies that
$(K:M)\subseteq Z(M) = \cup_{i=1} ^n p_i$. It follows from the
Prime Avoidance Theorem that $(K:M) \subseteq p_i$, for some $i$,
$1 \leq i \leq n$. Thus $K(Rm_i) = 0$, as desired.
\end{proof}

The remaining result of this paper provides the domination number
of the annihilating-submodule graph of a finite direct product of
modules.

\begin{thm}\label{c3.8} For a module $M$, which is a product of two
$($nonzero$)$
modules, one of the following holds:

\begin{itemize}
\item [(a)] If $M \cong F \times D$, where $F$ is a simple module
and $D$ is a prime module, then $\gamma(AG(M))=1$. \item [(b)] If
$M \cong D_1 \times D_2$, where $D_1$ and $D_2$ are prime modules
which are not simple, then $\gamma(AG(M))=2$.

\item [(c)] If $M \cong M_1 \times D$, where $M_1$ is a module
which is not prime and $D$ is a prime module, then $\gamma(AG(M))
= \gamma(AG(M_1)) + 1$.

\item [(d)]  If $M \cong M_1 \times M_2$, where $M_1$ and $M_2$
are two modules which are not prime, then $\gamma(AG(M)) =
\gamma(AG(M_1)) + \gamma(AG(M_2))$.

 \end{itemize}
\end{thm}

\begin{proof}
 Parts $(a)$ and $(b)$ are trivial.

 $(c)$ With no loss of generality, one can assume that $\gamma(AG(M_1)) < \infty$. Suppose that
$\gamma(AG(M_1)) =n$ and $\{K_1, \ldots , K_n \}$ is a minimal
dominating set of $AG(M_1)$. It is not hard to see that $\{K_1
\times 0, \ldots , K_n \times 0, 0\timesD\}$ is the smallest
dominating set of $AG(M)$.

$(d)$ We may assume that $\gamma(AG(M_1)) =m$ and $\gamma(AG(M_2))
=n$, for some positive integers $m$ and $n$. Let $\{K_1, \ldots ,
K_m\}$ and $\{L_1, \ldots , L_n\}$ be two minimal dominating sets
in $AG(M_1)$ and $AG(M_2)$, respectively. It is easy to see that
$\{K_1 \times 0, \ldots , K_m \times 0, 0 \times L_1 \ldots 0
\times L_n\}$ is the smallest dominating set in $AG(M)$.
\end{proof}

\end{document}